\documentclass[12pt,a4paper]{amsart}
\usepackage{amsfonts}
\numberwithin{equation}{section}

 \def\numset#1{{\mathbb #1}}
 \def\setZ{\numset{Z}}
\def\setR{\numset{R}}

\theoremstyle{plain}
\newtheorem{Th}{Theorem}[section]
\newtheorem{Lemma}[Th]{Lemma}

 \theoremstyle{definition}

\newtheorem{Conj}[Th]{Conjecture}
\newtheorem{Rem}{Remark}
\newtheorem{?}[Th]{Problem}

\newcommand{\kong}[2]{ \equiv #1 \pmod{#2}}

\newcommand{\eq}[1]{\eqref{#1}}
\newcommand{\xv}{{\underline x}}
     
\begin{document}

\title{More differences than multiple sums}

\author{Imre Z. Ruzsa}
\address{Alfr\'ed R\'enyi Institute of Mathematics\\
     Budapest, Pf. 127\\
     H-1364 Hungary
}
\email{ruzsa@renyi.hu}
 \thanks{Author was supported by ERC--AdG Grant No.321104  and
 Hungarian National Foundation for Scientific
 Research (OTKA), Grants No.109789 , % eny\'em
 and NK104183.}% Pintz            }

     \begin{abstract}
    We compare the size of the difference set $A-A$ to that of the set $kA$ of $k$-fold sums. We show the existence
of sets such that $|kA| < |A-A|^{\alpha_k}$ with $\alpha_k<1$. 
     \end{abstract}

     \maketitle

     \section{Introduction}

     The aim of this paper is to compare the size of the difference set $A-A$ and the size of
      \[ kA = A + \ldots +A, \ k \text{ times} \]
(we shall write $k \cdot A = \{ ka: a\in A \} $ for a set of multiples).

      Much has been written about the most natural case $k=2$. Freiman and Pigaev
\cite{freimanpigaev} proved that $|2A|^{3/4} \leq |A-A| \leq |2A|^{4/3}$. These are still the best exponents known,
though there is no reason to expect that either of them is exact. For other aspects and generalizations
see the papers \cite{r78a},\cite{r92e},\cite{hennecartetal99},\cite{r09b} and the books
\cite{nathanson96},\cite{taovu06},\cite{r09c}.

We will show the existence of sets of integers such that  $|kA| < |A-A|^{\alpha_k}$ with $\alpha_k<1$,
and of subsets of $\setZ_q$, the set of residues modulo $q$ for all sufficiently large $q$ such that
 \[ A-A = \setZ_q, \ |kA| < q^{\alpha_k} .  \]
 As far as I know, the only paper to deal with this problem is Haight's \cite{haight73}, who proved the existence,
for arbitrary prescribed positive integers $k$ and $l$, of a $q$ and a set $A\subset\setZ_q$ such that
$  A-A = \setZ_q $ and $kA$ avoids $l$ consecutive residues, and used this to show the existence of a set
$B$ of reals such that $B-B=\setR$ but $kB$ is of measure 0 for all $k$.

 Clearly if $|kA| < q^{\alpha_k}$, then there will
be gaps of size $ >  q^{1-\alpha_k}$, so the above result implies Haight's. On the other hand, it is not
difficult to deduce our result from Haight's either, so the two are essentially equivalent.
I also acknowledge that, while the details will be rather different, the main idea is taken from
Haight's paper. 

Haight's work remained rather unnoticed. A reason is that it was well ahead of its time,
before additive combinatorics became a fashionable subject; it is not an easy reading either.

In Section \ref{os} we shall consider the opposite question about the maximal possible size of
$kA$ compared to $A-A$.

\section{The main results}

We shall consider three ways of comparing sums and differences. For positive integers $k$ and $q$, $q>1$ write
 \[ F_k(q) = \min \{ |kA|: A\subset \setZ_q, A-A = \setZ_q \}, \]
\[ G_k(q) = \min \big\{ |kA|: A\subset \setZ, A-A \supset \{a+1, \ldots, a+q \} \text{ for some } a \big\}, \]
\[ H_k(q) = \min \{ |kA|: A\subset \setZ, |A-A| \geq q  \}. \]

Put
 \[ \alpha_k = \inf_{q\geq 2} \frac{\log G_k(q)}{ \log q} . \]

 \begin{Th}\label{mindegy}
   \[ \lim \frac{\log F_k(q)}{ \log q} =  \lim \frac{\log G_k(q)}{ \log q} =  \lim \frac{\log H_k(q)}{ \log q}  \]
    \[ =  \inf_{q\geq 2} \frac{\log H_k(q)}{ \log q} = \alpha_k. \]
 \end{Th}
One possible quantity is missing from the list.

\begin{?}
Is
 \[  \inf_{q\geq 2} \frac{\log F_k(q)}{ \log q} = \alpha_k \ ? \]
 \end{?}

 \begin{Th}\label{main}
    \[ 1-2^{-k} \leq \alpha_k <1 \]
for all $k$.   
 \end{Th}

The exact value is not know except the obvious $\alpha_1=1/2$. The bound $\alpha_2 \geq 3/4$ is Freiman and Pigaev's \cite{freimanpigaev}.
The upper bound from the construction below will be of type 1-1/tower.

\section{Properties of $F,G,H$}

We list some properties of these functions that together will imply Theorem \ref{mindegy}.

\begin{Lemma}\label{monoton} (Monotonicity.)
  If $q<q'$, then
   \[ G_k(q) \leq G_k(q'), \]
    \[ H_k(q) \leq H_k(q'). \]
\end{Lemma}

(Obvious, but important.)

\begin{?}
Is $F_k$ monotonically increasing?
 \end{?}

 \begin{Conj}
 No. Probably it depends on the multiplicative structure of $q$, not just its size.  
 \end{Conj}

\begin{Lemma}\label{submult}(Submultiplicativity.)
Let $q=q_1 q_2$. We have
\begin{equation}\label{Fsub}
F_k(q) \leq F_k(q_1) F_k(q_2) \text{ if } \gcd (q_1,q_2)=1, \end{equation}
\begin{equation}\label{Gsub}
G_k(q) \leq G_k(q_1) G_k(q_2) \text{ always,}  \end{equation}
\begin{equation}\label{Hsub}
H_k(q) \leq H_k(q_1) H_k(q_2) \text{ always.}  \end{equation}
\end{Lemma}

\begin{proof}
  Let $A_1, A_2$ be sets that give the value of our function for $q_1$ and $q_2$, resp.

To see \eq{Fsub} notice that $\setZ_q$ is isomorphic to the direct product $\setZ_{q_1} \times \setZ_{q_2}$, and the set
$A = A_1 \times A_2 $ gives the bound for  $F(q)$.

 To see \eq{Gsub} take the set $A=A_1 + q_1 \cdot A_2$.

To see \eq{Hsub} take the set $A=A_1 + m \cdot A_2$ with an integer $m$ chosen sufficiently large to avoid unwanted coincidences.
\end{proof}

\begin{?} Does $F_k(q) \leq F_k(q_1) F_k(q_2) $ hold for not coprime integers?
 \end{?}
 
 Monotonicity and submultiplicavity imply that
  \[ \lim \frac{\log G_k(q)}{ \log q} =\inf_{q\geq 2} \frac{\log G_k(q)}{ \log q} =\alpha_k, \
 \lim \frac{\log H_k(q)}{ \log q}    =  \inf_{q\geq 2} \frac{\log H_k(q)}{ \log q} . \]
To prove the other equalities in Theorem \ref{mindegy}
we show that these functions have the same order of magnitude.

\begin{Lemma}\label{viszony}
 For all $q$ we have
\begin{equation}\label{FG} F_k(q) \leq G_k(q), \end{equation}
\begin{equation}\label{HG} H_k(q) \leq G_k(q), \end{equation}
\begin{equation}\label{GF}  G_k(q) \leq G_k(2q+1) \leq 2kF_k(q). \end{equation}
\end{Lemma}
\begin{proof}
  Inequalities \eq{FG} and \eq{HG} are evident.

 To show \eq{GF}, let $A\subset\setZ_q$ be a set such that
  $A-A=\setZ_q$, $ | kA| = F_k(q)$. Define
   \[ A' = \{ n: -q < n \leq q, \ n \text{ \rm mod }q \in A \}.  \]
We claim that $A'-A'$ contains $2q+1$ consecutive integers, namely those in the  interval $[-q, q]$.
Indeed, if $-q \leq m \leq q$, then there are $x, y\in A'$, $1 \leq x,y \leq q$ such that
 \[ x-y \kong{m}{q} . \]
Consequently one of $x-y$, $(x-q)-y$, $x-(y-q)$ will be equal to $m$, and all are elements of $A'-A'$.
To estimate $|kA'|$ observe that
 \[ kA' \subset [-k(q-1), kq] . \]
This interval can be covered by $2k$ intervals of length $q$, and in each our set has at mos
$|kA|$ elements, hence  $|kA'| \leq 2k |kA|$.
\end{proof}

These results partially show Theorem \ref{mindegy}, except for the quantities involving $H_k$.
For $H_k$ we shall give the following estimate.

\begin{Lemma}\label{FHviszony}
\begin{equation}\label{FH} F_k(q) \leq c_k (\log q)^{k/2} H_k(q). \end{equation}
\end{Lemma}

The proof of this lemma is relegated to Section \ref{fir}. To prove our main result, Theorem \ref{main},
we shall work with $F$ and $G$; the results about $H$ are included because it is perhaps the most natural
quantity to consider.

\begin{?} Is $F_k(q) \leq c_k  H_k(q)$? Is $H_k(q) \leq F_k(q)$?
 \end{?}

\section{The construction}

In this section we prove that $\alpha_k<1$. 
We start by proving the following, seemingly weaker result.

\begin{Lemma}\label{gyengebb}
For every positive integer $k$ and positive $\varepsilon$ there is a positive integer $q$ and a set $A\subset\setZ_q$ such that 
$A-A=\setZ_q$, $|kA| < \varepsilon q$.
\end{Lemma}

\begin{proof}
  We shall describe our set in the form
   \[ A = \{ \varphi(x), x+\varphi(x): x\in\setZ_q \} \]
   via a function $\varphi: \setZ_q \to \setZ_q$. This guarantees $A-A=\setZ_q$.

   The set $kA$ is the collection of all elements of the form
   \begin{equation}\label{uv}
\sum_{x\in\setZ_q} \Big( u(x) \varphi(x) + v(x) \big(x+\varphi(x) \big)\Big) , \end{equation}
    where $u,v$ are nonnegative integer-valued functions on $\setZ_q$, satisfying
     \[ \sum_{x\in\setZ_q} \big( u(x)  + v(x) \big) =k. \]
     We define the \emph{level} of such a pair $(u,v)$ of functions as
      \[ l(u,v) = \#\{ x: u(x)+v(x)>0 \}  . \]
Clearly $1 \leq l(u,v) \leq k$.

For a function $\varphi$ and $1\leq m\leq k$, let $S_m(\varphi)$ denote the set of elements that have a representation
of the form \eq{uv}  with  $l(u,v) \leq m$ (in particular, $S_k(\varphi)=kA$). The construction will proceed recursively. First we
show how to find a modulus and a function such that $|S_1(\varphi)| <\delta q$. Next we show that, given two numbers
$\delta, \delta'$ such that $0<\delta<\delta'$,  a modulus and a function such that $|S_m(\varphi)| <\delta q$, we can find 
a modulus $q'$ and a corresponding function $\varphi'$ such that $|S_{m+1}(\varphi')| <\delta'q'$.

For the first step we will take a product of $k+1$ different primes, $q=p_0 \ldots p_k$ and identify $\setZ_q$ with the direct product
$\setZ_{p_0} \times \ldots \times \setZ_{p_k}$. We shall write elements of  $\setZ_q$ as vectors, $\xv=(x_0, \ldots, x_k)$, $x_i\in\setZ_{p_i}$. A pair $(u,v)$ of level 1
is supported by a single element $\xv$; necessarily $v(\xv)=k-u(\xv)$.  Hence elements of $S_1(\varphi)$ are of the form
 \[ u(\xv) \varphi(\xv) + \big(k-u(\xv)\big) \big(\xv+\varphi(\xv) \big) = k\varphi(\xv)+  \big(k-u(\xv)\big)\xv.\]
We will achieve that whenever $u(\xv)=j$, the $j$'th coordinate of this sum will vanish.
To this end we put
\[ \varphi (x_0, \ldots, x_k) = \left( -x_0, \frac{1-k}{k}x_1, \ldots,\frac{j-k}{k}x_j, \ldots, \frac{-1}{k}x_{k-1}, 0 \right). \]
Here division in the $j$'th coordinate is meant modulo $p_j$, and in order that this make sense we assume $p_j>k$
for all $j$.

The number of elements where the $j$'th coordinate vanishes is exactly $q/p_j$, consequently we have
 \[ |S_1(\varphi)| \leq q \sum \frac{1}{p_j} < \delta q \]
if we select primes so that $p_j>(k+1)/\delta$.
 
For the inductive step, assume that for some $1\leq m<k$ we are given two numbers
$\delta, \delta'$ such that $0<\delta<\delta'$,  a modulus $q$ and a function such that $|S_m(\varphi)| <\delta q$. We shall construct
a modulus $q'$ and a corresponding function $\varphi'$ such that $|S_{m+1}(\varphi')| <\delta'q'$.

Let $t$ be the number of pairs $(u,v)$ of level $m+1$ on $\setZ_q$. Our new number will be of the form
 \[ q' = q p_1 p_2 \ldots p_t, \]
with distinct primes $p_j$, not dividing $q$.  We identify $\setZ_{q'}$ with the direct product
$\setZ_q \times \setZ_{p_1} \times \ldots \times \setZ_{p_t}$. We shall write elements of $\setZ_{q'}$   as vectors, $\xv=(x_0,x_1 \ldots, x_t)$, $x_0\in\setZ_q$, $x_i\in\setZ_{p_i}$ for $i>0$.
The function $\varphi'$ will also be defined coordinatewise, as
 \[ \varphi'(\xv) = \big( \varphi_0(\xv), \ldots, \varphi_t(\xv) \big) . \]
We put $ \varphi_0(\xv) = \varphi(x_0)$. 

Given a pair $(u', v')$ on  $\setZ_q'$, we define its \emph{shadow} on $\setZ_q$ by the formula
 \[ u(x) = \sum_{x_1, \ldots, x_t} u'(x, x_1 \ldots, x_t), \   v(x) = \sum_{x_1, \ldots, x_t}  v'(x, x_1 \ldots, x_t) .     \]
Clearly  the level of $(u, v)$ does not exceed the level of $(u', v')$. 

 Elements of $S_{m+1}(\varphi')$ are of the form
 \begin{equation}\label{uv'}
\sum_{\xv\in\setZ_q'} \Big( u'(\xv) \varphi'(\xv) + v'(\xv) \big(\xv+\varphi'(\xv) \big)\Big) , \end{equation}
with pairs $(u', v')$ of level at most $m+1$. The 0'th coordinate of this sum is exactly
\[ \sum_{x\in\setZ_q} \Big( u(x) \varphi(x) + v(x) \big(x+\varphi(x) \big)\Big) , \]
where $(u,v)$ is the shadow of $(u',v')$. In particular, if the level of $(u,v)$ is at most $m$,
then the 0'th coordinate is an element of $S_m(\varphi)$.

Now we consider the case when the level of $(u,v)$, as well as of $(u',v')$, is $m+1$. Let
$(u_1, v_1), \ldots, (u_t, v_t)$ be a list of all pairs $(u,v)$ of level $m+1$. We shall achieve that
the $j$'th coordinate vanish whenever the shadow of $(u', v')$ is $(u_j, v_j)$.

Observe that the level of a pair $(u',v')$ and that of its shadow can be equal only if there is no coincidence
among the 0'th coordinate of those elements for which $u'(\xv)+v'(\xv)>0$; the sum defining the shadow has 
always at most one nonzero term. Consequently for all $\xv=(x_0,x_1 \ldots, x_t)$ either  $\bigl(u'(\xv), v'(\xv)\bigr)= (0,0)$
or  $\bigl(u'(\xv), v'(\xv)\bigr)= \bigl(u_j(x_0), v_j(x_0) \bigr)$. Thus all nonzero terms in the sum \eq{uv'} are of the form
 \[   u_j(x_0)  \varphi'(\xv) +   v_j(x_0)  \big(\xv+\varphi'(\xv) \big)  . \]
The $j$'th coordinate of this summand is
 \[   u_j(x_0)  \varphi_j(\xv) +   v_j(x_0)  \big(x_j+\varphi_j(\xv) \big)  . \]
 This will vanish if we define
  \[ \varphi_j(\xv) =
  \begin{cases}
    - \frac{v_j(x_0)}{u_j(x_0) + v_j(x_0)} & \text{ if } {u_j(x_0) + v_j(x_0)}>0, \\
    0  & \text{ if } {u_j(x_0) + v_j(x_0)}=0,
  \end{cases} \]
the division being understood modulo $p_j$.

This construction ensures that either the 0'th coordinate is in $S_m(\varphi)$ or another coordinate vanishes. Hence
 \[ \frac{ |S_{m+1}(\varphi')|}{q'} \leq  \frac{ |S_{m}(\varphi)|}{q} + \sum_{j=1}^t \frac{1}{p_j} < \delta +   \sum_{j=1}^t \frac{1}{p_j} < \delta',    \]
if we choose primes satisfying $p_j>t/(\delta'-\delta)$.

To prove the Lemma we start with $\delta= \varepsilon/(k+1)$ and proceed by finding moduli and functions with
$|S_m(\varphi)|/q<(m+1)\varepsilon/(k+1)$. After $k$ steps we have the desired bound for the size of $S_k(\varphi)=kA$.
\end{proof}

\begin{Rem}
  For the initial step I know several constructions, some of which yield smaller values of $q$; I chose this one
because it anticipates the inductive step.
\end{Rem}

\begin{proof}[Proof of Theorem \ref{main}.]
 
By virtue of Theorem \ref{mindegy}, to prove the upper bound it is sufficient to find 
\emph{a single $q$} such that $G_k(q)<q$; and by inequality \eq{GF}, it suffices to find 
a single $q$ such that $F_k(q)< q/k$, which is the previous lemma with $\varepsilon=1/k$.

To demonstrate the lower bound we show that for any finite set in any group we have
 \[ |kA| \geq |A-A|^{1- 2^{-k}} . \]
   We use induction on $k$. The case $k=1$ is evident. To go from $k$ to $k+1$ we use the inequality \cite{r78a}
(see also \cite{nathanson96},\cite{taovu06},\cite{r09c})
    \[  |X| |Y-Z| \leq |Y-X| |Y-Z|  \]
with $Y=Z=A$, $X=-kA$.
\end{proof}

\section{From integers to residues} \label{fir}

In this section we prove Lemma \ref{FHviszony}.

We start with an arbitrary set $A$ of integers, and in several steps we turn it into a set of residues modulo $q$.
Our tool will be the following projection-like transformation, which depends on a real parameter $t$:
 \[ \pi_t(n) =  [q\{ tn \} ]. \]
(We suppress the parameter $q$, which will be fixed through the section.) The values of $\pi_t$ are integers in $[1,q)$, and  $\pi_t$ 
has a quasi-additivity property:
 \[   \pi_t(x+y) =   \pi_t(x) + \pi_t(y) + r, \ r\in\{ 0,1, -q, 1-q \} . \]

\begin{Lemma}\label{vetites}
  Let $S$ be a set of integers, $|S| =q$. There is a $t\in(0,1)$ such that
   \[ | \pi_t(S)| > q/3 .\]
\end{Lemma}

\begin{proof}
  We select $t\in[0,1)$ randomly with uniform distribution. Let $z$ be the number of pairs $m,n\in S$
such that $  \pi_t(m) = \pi_t(n) $. For a fixed pair $(m,n)$ the probability that $  \pi_t(m) = \pi_t(n) $ is 1 if $m=n$,
and at most $2/q$ if $m \neq n$. To see the latter claim note that if  $  \pi_t(m) = \pi_t(n) $, then
\[ | \{ tm \} - \{ tn\}  | < 1/q, \]
hence $ \| t(m-n) \| <2/q$, which has probability $2/q$. Hence the expectation of $z$ is
 \[ \leq q + \frac{2}{q} q(q-1) < 3q .\]
Select any $t$ for which $z<3q$. For an integer $j\in[0,q)$ let $r_j$ be the number of integers $n\in S$ such that
$\pi_t(n) =j$. The inequality of arithmetic and square means yields
 \[ z = \sum r_j^2 \geq \frac{ \left(  \sum r_j \right)^2 }{ | \pi_t(S)|} =   \frac{ q^2 }{ | \pi_t(S)|} ,  \]
hence $ | \pi_t(S)| \geq q^2/z > q/3$ as wanted.
\end{proof}

\begin{Lemma}\label{Lorenz}
Let $A\subset\setZ_q$ be a nonempty set, $|A| \geq  tq$, $0<t<1$, and let $k$ be a positive integer. There are sets $B_1, \ldots, B_k \subset\setZ_q$
such that
 \[ A+B_1 + \ldots + B_k =  \setZ_q \]
 and
  \[ |B_i| \leq m = \left\lceil \left( \frac{\log q}{t} \right)^{1/k}  \right\rceil . \]
\end{Lemma}
\begin{proof}
Select $B_1$ randomly, with equal probability from all ${q \choose m}$ $m$-element subsets of $\setZ_q$. The probability that an
element of  $\setZ_q$ is not in $A+B_1$ is
 \[ {q-|A| \choose m} \Big/ {q \choose m} <(1-t)^m  .         \]
 Hence the expectation
of $ \left| \setZ_q \setminus (A+B_1) \right|$ is  $<(1-t)^mq$. Fix $B_1$ so that
 \[ \left| \setZ_q \setminus (A+B_1) \right| <(1-t)^mq . \]
Now repeat the process with $A+B_1$ in the place of $A$ to find $B_2$, and so on. After $k$ steps the number of elements 
outside $ A+B_1 + \ldots + B_k$ will be
 \[ <(1-t)^{m^k} q < e^{-tm^k} q < 1. \]
\end{proof}

The case $k=1$ of this lemma is  a theorem of Lorentz \cite{lorentz53} (see also \cite{sequences}).

\begin{proof}[Proof of  Lemma \ref{FHviszony}.]

Let $A$ be a set of integers such that $|A-A| \geq q$ and $|kA| = H_k(q)$. Put $A_1=\pi_t(A)$ with a number  $t$ such that
$| \pi_t(A-A)| > q/3$. The quasi-additivity property implies that
 \[ \pi_t(A-A)  \subset (A_1-A_1) + \{ 0,-1, q, q-1 \} . \]
Let $A_2\subset\setZ_q$ be the image of $A_1$. The above inclusion shows that $(A_2-A_2) + \{ 0,-1\} $ contains the image of $ \pi_t(A-A)$,
hence
 \[|A_2-A_2|  \geq | \pi_t(A-A)|/2 > q/6. \]
 Similarly $kA_2$ is contained in the image of  $ \pi_t(kA) + \{ 0, 1, \ldots, k-1 \}$, hence
  \[ | kA_2 | \leq k |kA|  = kH_k(q).\]
By Lemma \ref{Lorenz}, applied to the set $A_2$, there are sets $B_1, B_2\subset\setZ_q$ such that  $(A_2-A_2)+B_1+B_2=\setZ_q$ and $|B_i| < c \sqrt{\log q}$.
 Our set will be
$A_3 = A_2+(B_1 \cup -B_2)$. This set satisfies $A_3-A_3=\setZ_q$ and
 \[ |kA_3| \leq |kA_2| |k(B_1 \cup -B_2)   | < ( c \log q)^{k/2} |kA_2| \leq k( c \log q)^{k/2}  H_k(q). \]
\end{proof}

\section{The other side}\label{os}

In this section we  consider the  question about the maximal possible size of
$kA$ compared to $A-A$. Most results and proofs are completely analogous, and we shall not give details.

 For positive integers $k$ and $q$, $q>1$ write
 \[ f_k(q) = \min \{ |A-A|: A\subset \setZ_q, kA = \setZ_q \}, \]
\[ g_k(q) = \min \big\{ |A-A|: A\subset \setZ, kA \supset \{a+1, \ldots, a+q \} \text{ for some } a \big\}, \]
\[ h_k(q) = \min \{ |A-A|: A\subset \setZ, |kA| \geq q  \}. \]

Put
 \[ \beta_k = \inf_{q\geq 2} \frac{\log g_k(q)}{ \log q} . \]

 \begin{Th}\label{mindegy'}
   \[ \lim \frac{\log f_k(q)}{ \log q} =  \lim \frac{\log g_k(q)}{ \log q} =  \lim \frac{\log h_k(q)}{ \log q}  \]
    \[ =  \inf_{q\geq 2} \frac{\log h_k(q)}{ \log q} = \beta_k. \]
 \end{Th}

Again, I cannot decide whether
 \[  \inf_{q\geq 2} \frac{\log f_k(q)}{ \log q} = \beta_k. \]

The proof of this result proceeds through analogues of Lemmas \ref{monoton},  \ref{submult}, \ref{viszony}, \ref{FHviszony}.

\begin{Lemma}\label{monoton'} (Monotonicity.)
  If $q<q'$, then
   \[ g_k(q) \leq g_k(q'), \]
    \[ h_k(q) \leq h_k(q'). \]
\end{Lemma}

\begin{?}
Is $f_k$ monotonically increasing?
 \end{?}

 \begin{Conj}
 No. Probably it depends on the multiplicative structure of $q$, not just its size.  
 \end{Conj}

\begin{Lemma}\label{submult'}(Submultiplicativity.)
Let $q=q_1 q_2$. We have
\begin{equation}\label{fsub}
f_k(q) \leq f_k(q_1) f_k(q_2) \text{ if } \gcd (q_1,q_2)=1, \end{equation}
\begin{equation}\label{gsub}
g_k(q) \leq g_k(q_1) g_k(q_2) \text{ always,}  \end{equation}
\begin{equation}\label{hsub}
h_k(q) \leq h_k(q_1) h_k(q_2) \text{ always.}  \end{equation}
\end{Lemma}

\begin{Lemma}\label{viszony'}
 For all $q$ we have
\begin{equation}\label{fg} f_k(q) \leq g_k(q), \end{equation}
\begin{equation}\label{hg} h_k(q) \leq g_k(q), \end{equation}
\begin{equation}\label{gf}  g_k(q) \leq g_k(2q+1) \leq 4 f_k(q). \end{equation}
\end{Lemma}

\begin{Lemma}\label{fhviszony}
\begin{equation}\label{fh} f_k(q) \leq c_k (\log q)^{2/k} h_k(q). \end{equation}
\end{Lemma}

\begin{?} Is $f_k(q) \leq c_k  h_k(q)$? Is $h_k(q) \leq f_k(q)$?
 \end{?}

 \begin{Th}\label{main'}
   \begin{description}
     \item[(a)]
    \[  \frac{2}{k} - \frac{1}{k^2}    \leq \beta_k \leq \frac{2}{k} \]
for all $k$. 
\item[(b)]
$k\beta_k$ is increasing.  
\end{description}
 \end{Th}

 \begin{proof}
   For the lower estimate we show that
    \[ |kA| < |A-A|^{k^2/(2k-1)} \]
for every finite set in any commutative group. Write $|A| = n$, $|A-A| = tn$. By a Pl\"unnecke-type inequality (see e.g.
\cite{plunnecke70},\cite{r89e},\cite{nathanson96},\cite{taovu06},\cite{r09b})
 we get
\begin{equation}\label{plu}
|kA| \leq t^k n,\end{equation}
and obviously
\begin{equation}\label{triv}
|kA| < n^k . \end{equation}
By multiplying the $k$'th power of \eq{plu} and  $(k-1)$'th power of \eq{triv} and taking $k^2$'th root we get the desired bound.

For the upper estimate take a generic set without any coincidence among the $k$-fold sums.

Claim (b) follows from the fact that $|kA|^{1/k}$ is a decreasing function of $k$, see \cite{r10a}.
 \end{proof}

Claim (b) above leaves two possibilities: either  always $\beta_k<2/k$, or  $\beta_k=2/k$ after a point.

\begin{?}
Is always $\beta_k<2/k$?
 \end{?}

\begin{Conj}   Yes.
 \end{Conj}

As far as I know, the only known case is $k=2$. I think the case $k=4$ is particularly interesting:

\begin{?}
Is always $|4A| \leq |A-A|^2$?
 \end{?}

% \begin{Lemma}\label{x}
% \end{Lemma}

% \begin{equation}\label{y}
% \end{equation}

% \begin{?}
%  \end{?}

% \begin{Conj}   
%  \end{Conj}

   \bibliographystyle{amsplain}
     \bibliography{cimek,cikkeim}

   %   \begin{thebibliography}{9}

%      \end{thebibliography}

     \end{document}